\RequirePackage{fix-cm}
\RequirePackage{amsmath}
\documentclass[]{article}     
\usepackage{latexsym}
\usepackage{mathtools}
\mathtoolsset{showonlyrefs}
\usepackage{listings}
\usepackage{xcolor}
\usepackage{fullpage}
\usepackage{braket}
\usepackage{graphicx,epstopdf,amssymb,amsfonts,mathrsfs,enumitem,color,hyperref,url,bm}

\usepackage{algorithm,algorithmic}

\newcommand{\wt}{\widetilde}

\newcommand{\serie}[1]{\{#1_{n}\}_n}

\newcommand{\lowbiglquote}[1][18]{%
   \setbox0=\hbox{\fontsize{#1}{0}\selectfont``}%
   \setlength{\dimen0}{\ht0 - \heightof{A}}%
   \noindent\llap{\smash{\lower\dimen0\box0 }}}

\newcommand{\lowbigrquote}[1][18]{%
   \setbox0=\hbox{\fontsize{#1}{0}\selectfont''}%
   \setlength{\dimen0}{\ht0 - \heightof{A}}%
   \unskip\rlap{\smash{\lower\dimen0\box0 }}}

\newcommand{\ol}{\overline}

\newcommand{\ve}{\varepsilon}

\newcommand{\f}{\mathbb}
\newcommand{\cu}{\subseteq}
\newcommand{\GLT}{\sim_{GLT}}
\newcommand{\acs}{\xrightarrow{a.c.s.}}

\DeclareMathOperator{\diag}{diag}

\newtheorem{lemma}{Lemma}
\newtheorem{theorem}{Theorem}
\newtheorem{proposition}{Proposition}
\newtheorem{corollary}{Corollary}
\newtheorem{definition}{Definition}

\newtheorem{remark}{Remark}
\newtheorem{example}{Example}
\newtheorem{conjecture}{Conjecture}

\numberwithin{theorem}{section}

\title{Conjectures on spectral properties of ALIF algorithm}

\author{Giovanni Barbarino\thanks{Department of Mathematics and Systems Analysis, Aalto University, Finland.
              giovanni.barbarino@aalto.fi},
        Antonio Cicone\thanks{DISIM, Universit\`a degli Studi dell'Aquila, L'Aquila, Italy.
              antonio.cicone@univaq.it}
}

\begin{document}

\maketitle

\begin{abstract}
A new decomposition method for nonstationary signals, named Adaptive Local Iterative Filtering (ALIF), has been recently proposed in the literature. Given its similarity with the Empirical Mode Decomposition (EMD) and its more rigorous mathematical structure, which makes feasible to study its convergence compared to EMD, ALIF has really good potentiality to become a reference method in the analysis of signals containing strong nonstationary components, like chirps, multipaths and whistles, in many applications, like Physics, Engineering, Medicine and Finance, to name a few.

In \cite{cicone2019spectral}, the authors analyzed the spectral properties of the matrices produced by the ALIF method, in order to study its stability. Various results are achieved in that work through the use of Generalized Locally Toeplitz (GLT) sequences theory, a powerful tool originally designed to extract information on the asymptotic behavior of the spectra for PDE discretization matrices.
In this manuscript we focus on answering some of the open questions contained in \cite{cicone2019spectral}, and in doing so, we also develop new theory and results for the GLT sequences.

\end{abstract}


\section{Introduction}\label{sec:introduction}

The decomposition and subsequent time--frequency analysis of nonstationary signals is an important topic of research which received a significant acceleration from the publication of the seminal work on the Empirical Mode Decomposition (EMD) method by Huang et al.\ \cite{huang1998empirical} in 1998. In particular, Huang and his collegues at NASA proposed to iteratively decompose a given signal into a finite number of ``simple components'' called Intrinsic Mode Functions (IMFs) which fulfil the following two properties:
\begin{itemize}[nolistsep,leftmargin=*]
	\item the number of zero crossings and the number of extrema must be either equal or differ at most by one;
	\item at any point, the mean value of the envelope connecting the local maxima and the envelope connecting the local minima must be zero.
\end{itemize}

The decompositions produced using the EMD algorithm attracted the interest of a high number of researchers and it proved to be successful for a wide range of applications, as testified by the number of citations, more than 14600\footnote{Based on Scopus database}, that the paper \cite{huang1998empirical} by itself has received so far. Nevertheless, the EMD algorithm is based on the iterative calculation of envelopes which are taylored on the specific signal under study. This makes really hard to analyze the EMD mathematically. Furthermore, this  approach has also stability problems in the presence of noise, as illustrated in \cite{wu2009ensemble}. Several variants of the EMD have been recently proposed to address this last problem, e.g. the Ensemble Empirical Mode Decomposition (EEMD) \cite{wu2009ensemble}, the complementary EEMD \cite{yeh2010complementary}, the complete EEMD \cite{torres2011complete}, the partly EEMD \cite{zheng2014partly}, the noise assisted multivariate EMD (NA-MEMD) \cite{ur2011filter}. They all allow to address the EMD stability issue as well as to reduce the so called mode mixing problem \cite{zheng2014partly}. But their mathematical understanding, like the EMD one, is far from be complete. Furthermore, from the prospective of nonstationarities handling, they pose new challenges since they worsen the mode--splitting problem present in the EMD algorithm \cite{yeh2010complementary}.
Over the years many alternative approaches to the EMD have been proposed, like, for instance, the sparse TF representation \cite{hou2011adaptive,hou2009variant}, the Geometric mode decomposition \cite{yu2018geometric}, the Empirical wavelet transform \cite{gilles2013empirical}, the Variational mode decomposition \cite{dragomiretskiy2013variational}, and similar techniques \cite{selesnick2011resonance,meignen2007new,pustelnik2012multicomponent}. All these methods are based on optimization with respect to an a priori chosen basis. The only alternative method proposed so far in the literature which is based on iterations, and hence does not require any a priori assumption on the signal under analysis, is named Iterative Filtering \cite{lin2009iterative,cicone2019nonstationary,cicone2016adaptive}.  This alternative iterative method, although published only recently, has already been used effectively in a wide variety of applied fields, like, for instance, in \cite{yu2010modeling,an2016application,an2016wind,an2016demodulation,cicone2016hyperspectral,an2017local,an2017vibration,kim2016multiscale,yang2017oscillation,cicone2017howNonlinear,cicone2017Geophysics,sharma2017automatic,mitiche2018classification,li2018entropy,materassi2019stepping,sfarra2019improving,spogli2019role,spogli2019URSIrole,papini2020multidimensional,piersanti2020inquiry,ghobadi2020disentangling,ghobadi2020comparative}. The IF algorithm structure resembles the EMD one. Its key difference is in the way the signal moving average is computed, i.e., via correlation of the signal with an a priori chosen filter function, whereas, in the EMD--based methods, it is computed as average between two envelopes.
This apparently simple difference opens the doors to a complete mathematical analysis of the IF method \cite{huang2009convergence,cicone2016adaptive,cicone2017multidimensional,cicone2019MFIF,cicone2020iterative,cicone2020study,cicone2020numerical,cicone2020oneortwo,stallone2020new}. The only problem in the IF method is its limitation in the variability of the instantaneous frequency of each single IMF component. This becomes an issue when we are dealing with signals which contain strong nonstationarities, like the so called chirps and whistles. This is the mode--splitting problem which effects also EMD and derived algorithms \cite{yeh2010complementary}. To solve this problem, the Adaptive Local Iterative Filtering (ALIF) algorithm has been recently proposed in \cite{cicone2016adaptive}. ALIF is a flexible generalization of IF  which completely overcome the limitations of the IF method by computing wisely chosen local and adaptive signal averages.
This makes ALIF algorithm an extremely promising and unique technique for the extraction of chirps from nonstationary signals. However, the ALIF convergence cannot be guaranteed a priori yet. Some advances have been recently achieved in the literature \cite{cicone2019spectral,cicone2020convergence}, but the main questions are still open. In particular in \cite{cicone2019spectral} the authors propose two conjectures which we discuss thoroughly in this work.

The rest of this work is organized as follows. In Section \ref{sec:GLT} we recall all the basic mathematical tools required to analyze the ALIF iteration matrix asymptotic spectral properties. In Section \ref{sec:ALIF} we recall the ALIF methods and the conjectures originally proposed in \cite{cicone2019spectral}. Sections \ref{sec:Conj1} and \ref{sec:Conj2} are devoted  to the analysis of the two conjectures, for which we need some technical and auxiliary results reported and proved in the appendix.
In particular, Appendix \ref{sec:app} contains some novel contributions to the theory of GLT sequences and spectral symbols.

\section{Spectral analysis tools}\label{sec:GLT}
We present in this section the tools for analyzing the asymptotic spectral properties of the ALIF iteration matrix.
Throughout this paper, a matrix-sequence is any sequence of the form $\{A_n\}_n$, where $A_n$ is a square matrix of size $n$.

If $A$ is an $n\times n$ matrix and $1\le p\le\infty$, we denote by $\|A\|_p$ the Schatten $p$-norm of $A$, i.e., the $p$-norm of the vector $(\sigma_1(A),\ldots,\sigma_n(A))$ formed by the singular values of $A$. The Schatten $\infty$-norm $\|A\|_\infty$ is the largest singular value of $A$ and coincides with the spectral norm $\|A\|$. The Schatten 2-norm $\|A\|_2$ coincides with the Frobenius norm, i.e., $\|A\|_2=(\sum_{i,j=1}^n|a_{ij}|^2)^{1/2}$.

\subsection{Singular value and eigenvalue distribution of a matrix-sequence}
Let $C_c(\mathbb C)$ be the space of continuous complex-valued functions with bounded support defined on $\mathbb C$ and let $\mu_p$ be the Lebesgue measure in $\mathbb R^p$. If $A$ is a square matrix of size~$n$, the singular values and the eigenvalues of $A$ are denoted by $\sigma_1(A),\ldots,\sigma_n(A)$ and $\lambda_1(A),\ldots,\lambda_n(A)$, respectively.

\begin{definition}
	Let $\{A_n\}_n$ be a matrix-sequence and let $f:D\subset\mathbb R^p\to\mathbb C$ be a measurable function defined on a set $D$ with $0<\mu_p(D)<\infty$.
	\begin{itemize}[nolistsep,leftmargin=*]
		\item We say that $\{A_n\}_n$ has a singular value distribution described by $f$, and we write $\{A_n\}_n\sim_\sigma f$, if for all $F\in C_c(\mathbb C)$ we have
		\[ \lim_{n\to\infty}\frac1{n}\sum_{i=1}^{n}F(\sigma_i(A_n))=\frac1{\mu_p(D)}\int_DF(|f(y_1,\ldots,y_p)|){\rm d}y_1\ldots{\rm d}y_p. \]
		\item We say that $\{A_n\}_n$ has an eigenvalue distribution described by $f$, and we write $\{A_n\}_n\sim_\lambda f$, if for all $F\in C_c(\mathbb C)$ we have
		\[ \lim_{n\to\infty}\frac1{n}\sum_{i=1}^{n}F(\lambda_i(A_n))=\frac1{\mu_p(D)}\int_DF(f(y_1,\ldots,y_p)){\rm d}y_1\ldots{\rm d}y_p. \]
	\end{itemize}
	If $\{A_n\}_n$ has both a singular value and an eigenvalue distribution described by $f$, we write $\{A_n\}_n\sim_{\sigma,\lambda}f$.
\end{definition}

\subsection{Informal meaning of the singular value and eigenvalue distribution}
Assuming $f$ is Riemann-integrable, the eigenvalue distribution $\{A_n\}_n\sim_\lambda f$ has the following informal meaning \cite[Section~3.1]{GLT-book}: all the eigenvalues of $A_n$, except possibly for $o(n)$ outliers, are approximately equal to the samples of $f$ over a uniform grid in $D$ (for $n$ large enough).
For instance, if $p=1$ and $D=[a,b]$, then, assuming we have no outliers, the eigenvalues of $A_n$ are approximately equal to
\[ f\Bigl(a+i\,\frac{b-a}{n}\Bigr),\quad i=1,\ldots,n, \]
for $n$ large enough. Similarly, if $p=2$, $n=m^2$ and $D=[a_1,b_1]\times [a_2,b_2]$, then, assuming we have no outliers, the eigenvalues of $A_n$ are approximately equal to
\[ f\Bigl(a_1+i\,\frac{b_1-a_1}m,\,\,a_2+j\,\frac{b_2-a_2}m\Bigr),\quad i,j=1,\ldots,m, \]
for $n$ large enough. A completely analogous meaning can also be given for the singular value distribution $\{A_n\}_n\sim_\sigma f$.

\subsection{Zero-distributed sequences}
A matrix-sequence $\{Z_n\}_n$ such that $\{Z_n\}_n\sim_\sigma0$ is referred to as a zero-distributed sequence. In other words, $\{Z_n\}_n$ is zero-distributed if and only if $\lim_{n\to\infty}\frac1{n}\sum_{i=1}^nF(\sigma_i(Z_n))=F(0)$ for all $F\in C_c(\mathbb C)$.
Proposition~\ref{res:zero-dis} is proved in \cite[Section~3.4]{GLT-book} and provides an important characterization of zero-distributed sequences together with a useful sufficient condition for detecting such sequences. For convenience, throughout this paper we use the natural convention $1/\infty=0$.
\begin{proposition}\label{res:zero-dis}
	Let $\{Z_n\}_n$ be a matrix-sequence.
	\begin{itemize}[nolistsep,leftmargin=*]
		\item $\{Z_n\}_n$ is zero-distributed if and only if $Z_n=R_n+N_n$ with
		\[ \lim_{n\to\infty}n^{-1}{\rm rank}(R_n)=\lim_{n\to\infty}\|N_n\|=0. \]
		\item $\{Z_n\}_n$ is zero-distributed if there is a $p\in[1,\infty]$ such that
		\[ \lim_{n\to\infty}n^{-1/p}\|Z_n\|_p=0. \]
	\end{itemize}
\end{proposition}

\subsection{Sequences of diagonal sampling matrices}\label{sub:diag}
If $n\in\mathbb N$ and $a:[0,1]\to\mathbb C$, the $n$th diagonal sampling matrix generated by $a$ is the $n\times n$ diagonal matrix given by
\begin{equation*}
D_n(a)=\mathop{\rm diag}_{i=1,\ldots,n}a\Bigl(\frac in\Bigr).
\end{equation*}
$\{D_n(a)\}_n$ is called the sequence of diagonal sampling matrices generated by $a$.

\subsection{Toeplitz sequences}\label{sub:toep}
If $n\in\mathbb N$ and $f:[-\pi,\pi]\to\mathbb C$ is a function in $L^1([-\pi,\pi])$, the $n$th Toeplitz matrix generated by $f$ is the $n\times n$ matrix
\[ T_n(f)=[\hat f_{i-j}]_{i,j=1}^n=\begin{bmatrix}
\hat f_0 & \hat f_{-1} & \ \hat f_{-2} & \ \cdots & \ \ \cdots & \hat f_{-(n-1)} \\
\hat f_1 & \ddots & \ \ddots & \ \ddots & \ \ & \vdots\\
\hat f_2 & \ddots & \ \ddots & \ \ddots & \ \ \ddots & \vdots\\
\vdots & \ddots & \ \ddots & \ \ddots & \ \ \ddots & \hat f_{-2}\\
\vdots & & \ \ddots & \ \ddots & \ \ \ddots & \hat f_{-1}\\
\hat f_{n-1} & \cdots & \ \cdots & \ \hat f_2 & \ \ \hat f_1 & \hat f_0
\end{bmatrix}, \]
where the numbers $\hat f_k$ are the Fourier coefficients of $f$,
\[ \hat f_k=\frac1{2\pi}\int_{-\pi}^\pi f(\theta){\rm e}^{-{\rm i}k\theta}{\rm d}\theta,\quad k\in\mathbb Z. \]
$\{T_n(f)\}_n$ is called the Toeplitz sequence generated by $f$.

\subsection{Approximating classes of sequences}
The notion of approximating classes of sequences (a.c.s.)\ is fundamental to the theory of Generalized Locally Toeplitz (GLT) sequences and it is deeply studied in \cite[Chapter~5]{GLT-book}.

\begin{definition}\label{a.c.s.}
	Let $\{A_n\}_n$ be a matrix-sequence and let $\{\{B_{n,m}\}_n\}_m$ be a sequence of matrix-sequences. We say that $\{\{B_{n,m}\}_n\}_m$ is an approximating class of sequences (a.c.s.)\ for $\{A_n\}_n$ if the following condition is met: for every $m$ there exists $n_m$ such that, for $n\ge n_m$,
	\begin{equation*}\label{splitting}
	A_n=B_{n,m}+R_{n,m}+N_{n,m},\quad {\rm rank}(R_{n,m})\le c(m)n,\quad \|N_{n,m}\|\le\omega(m),
	\end{equation*}
	where $n_m,\,c(m),\,\omega(m)$ depend only on $m$, and $\displaystyle\lim_{m\to\infty}c(m)=\lim_{m\to\infty}\omega(m)=0$.
\end{definition}

Roughly speaking, $\{\{B_{n,m}\}_n\}_m$ is an a.c.s.\ for $\{A_n\}_n$ if, for large $m$, the sequence $\{B_{n,m}\}_n$ approximates $\{A_n\}_n$ in the sense that $A_n$ is eventually equal to $B_{n,m}$ plus a small-rank matrix (with respect to the matrix size $n$) plus a small-norm matrix. We will use the convergence notation $\{B_{n,m}\}_n\stackrel{\rm a.c.s.}{\longrightarrow}\{A_n\}_n$
to indicate that $\{\{B_{n,m}\}_n\}_m$ is an a.c.s.\ for $\{A_n\}_n$.
A useful criterion to test the a.c.s.\
convergence is provided in the next theorem \cite[Corollary~5.3]{GLT-book}.

\begin{theorem}\label{res:acs_crit}
	Let $\{A_n\}_n,\{B_{n,m}\}_n$ be sequences of matrices, with $A_n,B_{n,m}$ of size $n$, and let $1\le p<\infty$. Suppose that for every $m$ there exists $n_m$ such that, for $n\ge n_m$,
	\[ \|A_n-B_{n,m}\|_p^p\le\varepsilon(m,n)n, \]
	where $\displaystyle\lim_{m\to\infty}\limsup_{n\to\infty}\varepsilon(m,n)=0$. Then $\{B_{n,m}\}_n\stackrel{\rm a.c.s.}{\longrightarrow}\{A_n\}_n.$
\end{theorem}

\subsection{GLT sequences}
A GLT sequence $\{A_n\}_n$ is a special matrix-sequence equipped with a measurable function $\kappa:[0,1]\times[-\pi,\pi]\to\mathbb C$, the so-called symbol (or kernel). We use the notation $\{A_n\}_n\sim_{\rm GLT}\kappa$ to indicate that $\{A_n\}_n$ is a GLT sequence with symbol $\kappa$.
The properties of GLT sequences that we shall need in this paper are listed below; the corresponding proofs can be found in \cite{pert,GLT-book}.
\begin{enumerate}[leftmargin=39pt]
	\item[\textbf{GLT\,1.}] If $\{A_n\}_n\sim_{\rm GLT}\kappa$ then $\{A_n\}_n\sim_\sigma\kappa$. If $\{A_n\}_n\sim_{\rm GLT}\kappa$ and the matrices $A_n$ are Hermitian then $\{A_n\}_n\sim_\lambda\kappa$.
	\item[\textbf{GLT\,2.}]
	 If $\{A_n\}_n\sim_{\rm GLT}\kappa$ and $A_n=X_n+Y_n$, where
	\begin{itemize}
		\item every $X_n$ is Hermitian,
		\item $n^{-1/2}\|Y_n\|_2\to 0$,
	\end{itemize}
	then $\{A_n\}_n\sim_{\lambda}\kappa$.
	\item[\textbf{GLT\,3.}] We have
	\begin{itemize}[nolistsep,leftmargin=*]
		\item $\{T_n(f)\}_n\sim_{\rm GLT}\kappa(x,\theta)=f(\theta)$ if $f\in L^1([-\pi,\pi])$,
		\item $\{D_n(a)\}_n\sim_{\rm GLT}\kappa(x,\theta)=a(x)$ if $a:[0,1]\to\mathbb C$ is Riemann-integrable,
		\item $\{Z_n\}_n\sim_{\rm GLT}\kappa(x,\theta)=0$ if and only if $\{Z_n\}_n\sim_\sigma0$.
	\end{itemize}
	\item[\textbf{GLT\,4.}] If $\{A_n\}_n\sim_{\rm GLT}\kappa$ and $\{B_n\}_n\sim_{\rm GLT}\xi$ then
	\begin{itemize}[nolistsep,leftmargin=*]
		\item $\{A_n^*\}_n\sim_{\rm GLT}\overline\kappa$,
		\item $\{\alpha A_n+\beta B_n\}_n\sim_{\rm GLT}\alpha\kappa+\beta\xi$ for all $\alpha,\beta\in\mathbb C$,
		\item $\{A_nB_n\}_n\sim_{\rm GLT}\kappa\xi$.
	\end{itemize}
	\item[\textbf{GLT\,5.}] $\{A_n\}_n\sim_{\rm GLT}\kappa$ if and only if there exist GLT sequences $\{B_{n,m}\}_n\sim_{\rm GLT}\kappa_m$ such that $\{B_{n,m}\}_n\stackrel{\rm a.c.s.}{\longrightarrow}\{A_n\}_n$ and $\kappa_m\to\kappa$ in measure over $[0,1]\times[-\pi,\pi]$.
\end{enumerate}
If $\{A_n\}_n$ has singular value, eigenvalue distribution and GLT symbol described by a single function $\kappa$, we write $\{A_n\}_n\sim_{GLT,\sigma,\lambda}\kappa$.
\section{The ALIF method}\label{sec:ALIF}

\subsection{Terminology}\label{sec:term}
Throughout this paper, any real function $g:\mathbb R\to\mathbb R$ is also referred to as a signal. Without loss of generality, we assume that the domain on which every signal $g$ is studied is the reference interval $[0,1]$. Outside the reference interval, the signal is usually not known and so, whenever necessary, we have to make assumptions, that is, we have to impose boundary conditions. The extrema of a signal $g$ are the points belonging to $(0,1)$ where $g$ attains its local maxima and minima.
If $\mathbf g=[\mathbf g_0,\ldots,\mathbf g_{n-1}]$ is a vector in $\mathbb R^n$, the extrema of $\mathbf g$ are the indices belonging to $\{1,\ldots,n-2\}$ where $\mathbf g$ attains its local maxima and minima, i.e., the indices $j\in\{1,\ldots,n-2\}$ such that $\mathbf g_j>\max(\mathbf g_{j-1},\mathbf g_{j+1})$ or $\mathbf g_j<\min(\mathbf g_{j-1},\mathbf g_{j+1})$.
A filter $k$ is an even, nonnegative, bounded, measurable,\footnote{Throughout this paper, the word ``measurable'' always means ``Lebesgue measurable''.} and compactly supported function from $\mathbb R$ to $\mathbb R$ satisfying the normalization condition $\int_{\mathbb R}k(y){\rm d}y=1$.
  We refer to $\ell=\sup\{y>0:\,k(y)>0\}$ as the length of the filter $k$. Note that $0<\ell<\infty$ and the support of $k$ is contained in $[-\ell,\ell]$.

\subsection{The ALIF method}\label{sec:ALIFc}

\begin{algorithm}
	\caption{\textbf{(ALIF Algorithm)} ${\rm IMFs=ALIF}(g)$}
	\begin{algorithmic}
		\STATE IMFs = $\left\{\right\}$
		\STATE initialize the remaining signal $r=g$
		\WHILE{the number of extrema of $r$ is $\geq 2$}
		\STATE for each $x\in[0,1]$ compute the filter $k_x$, whose length $\ell(x)$ changes from $x$ to $x$ based on $r$ itself
		\STATE $g_1=r$
		\STATE $m=1$
		\WHILE{the stopping criterion is not satisfied}
		\STATE compute the moving average $f_m$ of the signal $g_m$ as\\
		$f_m(x)=\int_{\mathbb R}g_m(y)k_x(x-y){\rm d}y$
		\STATE $g_{m+1}=g_m-f_m$
		\STATE $m=m+1$
		\ENDWHILE
		\STATE ${\rm IMFs}={\rm IMFs}\cup\{g_m\}$
		\STATE $r=r-g_m$
		\ENDWHILE
	\end{algorithmic}
	\label{alg:ALIF}
\end{algorithm}

As mentioned in Section~\ref{sec:introduction}, the ALIF method is an iterative procedure whose purpose is to decompose a signal $g$ into a finite number of ``simple components'', the so-called IMFs of $g$.
Algorithm~\ref{alg:ALIF} shows the pseudocode of the ALIF method, in which the input is a signal $g$ and the output is the set of the IMFs of $g$. The ALIF algorithm contains two loops. The inner loop captures a single IMF, while the outer loop produces all the IMFs embedded in $g$. Considering the first iteration of the ALIF outer loop in which $g_1=g$, we see that the key idea to extract the first IMF consists in computing the moving average of $g_m$ and subtract it from $g_m$ itself so as to capture the fluctuation part $\mathcal S_m(g_m)=g_m-f_m=g_{m+1}$. This is repeated iteratively and, assuming convergence, the first IMF is obtained as ${\rm IMF}_1=\lim_{m\to\infty}\mathcal S(g_m)$. In practice, however, we cannot let $m$ go to $\infty$ and we have to use a stopping criterion, as indicated in Algorithm~\ref{alg:ALIF}. Assuming convergence, one can stop the inner loop at the first index $m$ such that the difference $g_{m+1}-g_m$ is small in some norm (possibly, a norm for which the convergence is known). A safer stopping criterion also imposes a limit on the maximum number of iterations. This method of IMF extraction is shared with the EMD and IF algorithms and the only difference consists in the computation of the moving average. In the ALIF method, it is computed through the convolution with a filter $k_x$, that can depend on the point $x$. In practical applications of the ALIF method, first a length function $\ell(x)$ is computed based on the signal $g_1$, and then $k_x$ is chosen as\,\footnote{Note that $k_x$ in \eqref{k-choice'} is indeed a filter according to the terminology introduced in Section~{\rm \ref{sec:term}}.}
	\begin{equation}\label{k-choice'}
	k_x(y)=\frac{k\bigl(\frac{y}{\ell(x)}\bigr)}{\ell(x)},
	\end{equation}
	where $k$ is an a priori fixed filter with length $1$, so that the length of $k_x$ is $\ell(x)$.
 Once the first IMF is obtained, to produce the second IMF we apply the previous process to the remaining signal $r=g-{\rm IMF}_1$. We then iterate this procedure to obtain all the IMFs of $g$, and we stop as soon as the remaining signal becomes a trend signal, meaning that it possesses at most one extremum. Clearly, the sum of all the IMFs of $g$ produced by the ALIF method with the final trend signal $r$ is equal to $g$.

\begin{remark}\label{ell(x)}
	In the case where $\ell(x)$ is chosen at each iteration of the outer loop as a constant $\ell$, depending on the remaining signal $r$ but not on $x$, the ALIF method reduces to the IF method, whose convergence has been studied in~{\rm\cite{cicone2019spectral,huang2009convergence}}.
\end{remark}


\subsection{The Discrete ALIF method}\label{sec:ALIFd}
In practice, we usually do not know a signal $g$ on the whole reference interval $[0,1]$. What we actually know are the samples of $g$ over a fine grid in $[0,1]$. We therefore need a discrete version of the ALIF algorithm, which is able to (approximately) capture the IMFs of $g$ by exploiting this sole information.
From now on, we make the following assumptions.
\begin{itemize}[nolistsep,leftmargin=*]
	\item For any signal $g$, no other information about $g$ is available except for its samples at the $n$ points $x_i=\frac{i}{n-1}$, $i=0,\ldots,n-1$. Moreover, $g=0$ outside $[0,1]$ (so we are imposing homogeneous Dirichlet boundary conditions).
	\item The filter $k_x$ is defined as in \eqref{k-choice'} in terms of an a priori fixed filter $k$ with length 1.
\end{itemize}
Under these hypotheses, what we may ask to a discrete version of the ALIF algorithm is to compute the (approximated) samples of the IMFs of $g$ at the sampling points $x_i$, $i=0,\ldots,n-1$. This is done by approximating the moving average at the points $x_i$ through the rectangle formula or any other quadrature rule. Setting for convenience $x_j=\frac{j}{n-1}$ for all $j\in\mathbb Z$, the rectangle formula yields the approximation
\begin{equation*}
f_m(x_i)=\int_{\mathbb R}g_m(y)k_{x_i}(x_i-y){\rm d}y\approx\frac{1}{n-1}\sum_{j\in\mathbb Z}g_m(x_j)k_{x_i}(x_i-x_j),\quad i=0,\ldots,n-1,
\end{equation*}
where we note that the sum is finite because $k_{x_i}$ is compactly supported. Assuming that, at each iteration of the ALIF inner loop, the signal $g_m$ is set to zero outside the reference interval $[0,1]$, the previous equation becomes
\begin{equation*}
f_m(x_i)\approx\frac{1}{n-1}\sum_{j=0}^{n-1}g_m(x_j)k_{x_i}(x_i-x_j),\quad i=0,\ldots,n-1.
\end{equation*}
We then obtain
\begin{align}\label{ad2}
g_{m+1}(x_i)&=g_m(x_i)-f_m(x_i)\notag\\
&\approx g_m(x_i)-\frac1{n-1}\sum_{j=0}^{n-1}g_m(x_j)k_{x_i}(x_i-x_j),\quad i=0,\ldots,n-1.
\end{align}
Denoting by $\mathbf g=[g(x_0),\ldots,g(x_{n-1})]^T$ the vector containing the samples of the signal $g$ at the sampling points $x_i$, we can rewrite \eqref{ad2} in matrix form as follows:
\begin{equation}\label{replacer}
\mathbf g_{m+1}\approx (I_n-K_n)\mathbf g_m,
\end{equation}
where $I_n$ is the $n\times n$ identity matrix and
\begin{equation}\label{ALIF_matrix}
K_n\hspace{-1pt}=\hspace{-1pt}\left[\frac1{n-1}\,k_{x_i}(x_i-x_j)\right]_{i,j=0}^{n-1}\hspace{-2pt}=\hspace{-1pt}\left[\frac{k\bigl(\frac{x_i-x_{j_{\vphantom{1}}}}{\ell(x_i)}\bigr)}{(n-1)\ell(x_i)}\right]_{i,j=0}^{n-1}\hspace{-2pt}=\hspace{-1pt}\left[\frac{k\bigl(\frac{i-j}{(n-1)\ell(x_i)}\bigr)}{(n-1)\ell(x_i)}\right]_{i,j=0}^{n-1}.
\end{equation}

\begin{algorithm}
	\caption{\textbf{(Discrete ALIF Algorithm)} $\mathbf{IMFs}={\rm ALIF}(\mathbf g)$}
	\begin{algorithmic}
		\STATE $\mathbf{IMFs} = \left\{\right\}$
		\STATE initialize the remaining signal $\mathbf r=\mathbf g$
		\WHILE{the number of extrema of $\mathbf r$ is $\geq 2$}
		\STATE for each $x_i=x_0,\ldots,x_{n-1}$ compute the filter $k_{x_i}$, whose length $\ell(x_i)$ changes from $x_i$ to $x_i$ based on $\mathbf r$ itself
		\STATE $\mathbf g_1=\mathbf r$
		\STATE $m=1$
		\WHILE{the stopping criterion is not satisfied}
		\STATE extend $\mathbf g_m$ to $\mathbb Z$ by setting $(\mathbf g_m)_j=0$ for $j\not\in\{0,\ldots,n-1\}$
		\STATE compute the moving average $\mathbf f_m$ of $\mathbf g_m$ as\\
		$(\mathbf f_m)_i=\frac{1}{n-1}\sum_{j\in\mathbb Z}(\mathbf g_m)_jk_{x_i}(x_i-x_j),\ i=0,\ldots,n-1$
		\STATE $\mathbf g_{m+1}=\mathbf g_m-\mathbf f_m$
		\STATE $m=m+1$
		\ENDWHILE
		\STATE $\mathbf{IMFs}=\mathbf{IMFs}\cup\{\mathbf g_m\}$
		\STATE $\mathbf r=\mathbf r-\mathbf g_m$
		\ENDWHILE
	\end{algorithmic}
	\label{alg:ALIFd}
\end{algorithm}

The pseudocode for the Discrete ALIF method\,\footnote{Available at \url{www.cicone.com}.} is reported in Algorithm~\ref{alg:ALIFd}. The input is a vector $\mathbf g=[\mathbf g_0,\ldots,\mathbf g_{n-1}]^T=[g(x_0),\ldots,g(x_{n-1})]^T$ containing the samples of a signal $g$ at the sampling points $x_i=\frac{i}{n-1}$, $i=0,\ldots,n-1$, while the output is the set of vectors containing the (approximated) samples of the IMFs of $g$ at the same points $x_i$. Note that the first four lines inside the inner loop of Algorithm~\ref{alg:ALIFd} can be replaced by the sole equation $\mathbf g_{m+1}=(I_n-K_n)\mathbf g_m$, which is obtained from \eqref{replacer} by turning ``\,$\approx$\,'' into ``\,$=$\,''. Assuming convergence, the vector $\mathbf{IMF}_1$ containing the (approximated) samples of the first IMF is obtained as $\mathbf{IMF}_1=(I_n-K_n)^m\mathbf r$ with $\mathbf r=\mathbf g$ and $m$ large enough so that the stopping criterion is met. Similarly, $\mathbf{IMF}_2=(I_n-K_n)^m\mathbf r$ with $\mathbf r=\mathbf g-\mathbf{IMF}_1$ and $m$ large enough, $\mathbf{IMF}_3=(I_n-K_n)^m\mathbf r$ with $\mathbf r=\mathbf g-\mathbf{IMF}_1-\mathbf{IMF}_2$ and $m$ large enough, etc.
Note that the matrix $K_n$ used to compute $\mathbf{IMF}_i$ is different in general from the matrix $K_n$ used to compute $\mathbf{IMF}_j$ if $i\ne j$. Indeed, the matrix $K_n$ changes at every iteration of the outer loop because, although the filter $k$ is fixed, the length $\ell(x_i)$ depends on the remaining signal $\mathbf r$ and changes with it.

\begin{remark}\label{cc}
	A necessary condition for the convergence of the Discrete ALIF method is that
	\begin{equation}\label{cn}
	|1-\lambda_i(K_n)|\le1,\quad i=1,\ldots,n.
	\end{equation}
	Indeed, if \eqref{cn} is violated then $\rho(I_n-K_n)>1$ and $(I_n-K_n)^m\mathbf r$ diverges to $\infty$ (with respect to any norm of $\mathbb R^n$) for almost every vector $\mathbf r\in\mathbb R^n$.
\end{remark}

\subsection{Conjectures}
The ALIF iteration matrix $K_n$ is thus described by
\[
K_n = \frac{1}{n-1}\left[ \frac{k((x_i-x_j)/l_n(x_i))}{l_n(x_i)}  \right]_{i,j=0}^{n-1}
\]
where $x_i=\frac i{n-1}$ and $k,l_n$ are real-valued functions. Moreover $k$ is an even, non-negative, bounded, compactly supported measurable function with $\|k\|_1 =1$. We always consider $k$ of length 1, meaning that it is supported on $[-1,1]$, and we take $l_n(x)$ strictly positive on $[0,1]$.
From now on,  we always suppose that $L(x) := (n-1)l_n(x)$ is independent of $n$, so that we can rewrite $K_n$ as
\[
K_n = \left[ \frac{k((i-j)/L(x_i))}{L(x_i)}  \right]_{i,j=0}^{n-1}.
\]
In this case, it is possible to analyze the asymptotic spectral properties of the sequence $\serie K$ through the use of GLT theory introduced in Section \ref{sec:GLT}. If we denote
\[
\kappa(x,\theta)
:= \frac{1}{L(x)}\sum_{j\in\f Z}k\left(\frac{j}{L(x)}\right) e^{\textnormal{i} j\theta},
\]
then it is possible to come up with the following result.
\begin{lemma}[\cite{cicone2019spectral}]\label{res:known}
	Suppose that one of the following hypotheses is satisfied:
	\begin{itemize}
		\item  $L(x)$ is a step function,
		\item $k(x),L(x)$ are continuous functions with $L(x)\ge L_*>0$.
	\end{itemize}
	In this case,
	\[
	\serie K \sim_{GLT,\sigma,\lambda} \kappa(x,\theta).
	\]
\end{lemma}

Notice that $\kappa(x,\theta)$ is a real valued function, since $k$ is an even function. From Remark \ref{cc},  the necessary condition for the convergence of the Discrete ALIF method  can be written as follows:
\begin{equation}\label{eq:necessary_condition}
0\le \kappa(x,\theta) \le 2, \quad  (x,\theta) \in [0,1]\times[-\pi,\pi].
\end{equation}
Here we report the two conjectures from \cite{cicone2019spectral} that suggest how to generalize Lemma \ref{res:known} and that \eqref{eq:necessary_condition} may be actually a sufficient condition for the convergence of the ALIF method.

\begin{conjecture}\label{conj:ALIF_1}
	Suppose   that
	\[
	f_j(x) := \frac{k(j/L(x))}{L(x)}\text{  is Riemann-integrable over }[0,1]\quad \forall\,j\in \f Z.
	\]
	Then, for the sequence of ALIF iteration matrices $\serie K$,
	\[
	\serie K \sim_{GLT,\sigma,\lambda} \kappa(x,\theta).
	\]
\end{conjecture}

\begin{conjecture}\label{conj:ALIF_2}
	Assuming the hypotheses of Lemma \ref{res:known}, and \eqref{eq:necessary_condition}, the Discrete ALIF converges.
\end{conjecture}

In the next sections we discuss both the conjectures, developing new tools to answer and analyze the questions.

\section{Conjecture 1}\label{sec:Conj1}

In Section \ref{sub:diag} and \ref{sub:toep}, we have introduced the fundamental GLT sequences $\{D_n(a)\}_n$ referred to a Riemann-integrable function $a$, and $T_n(f)$ referred to an $L^1$ function $f$.
 From
\begin{equation}\label{eq:diag'}
	K_n = \left[ \frac{k((i-j)/L(x_i))}{L(x_i)}  \right]_{i,j=0}^{n-1},\qquad
	f_p(x) := \frac{k(p/L(x))}{L(x)},
	\qquad
	D'_n(f_p):= \diag([f_p(x_i)]_{i=0,\dots,n-1})
\end{equation}
one can easily verify that the ALIF iteration matrix $K_n$ can be rewritten as
\[
K_n =  \sum_{p\in \f Z}D'_n(f_p)T_n(e^{\textnormal i p\theta}).
\]
The diagonal matrix
$D_n'(f_p)$
differs from $D_n(f_p)$ only because we are considering a different regular grid of points where to evaluate the function $f_p$. Anyway, it is possible to prove that the sequence $\{D_n(f_p) - D_n'(f_p) \}_n$ is zero-distributed whenever $f_p$ is Riemann-Integrable, so,
thanks to \textbf{GLT 3} and \textbf{GLT 4}, we can say that $\{D_n'(f_p) \}_n\GLT f_p$ and
\[
K_{n,m} := \sum_{p=-m}^{m} D_n'(f_p)T_n(e^{\textnormal{i} p\theta})\implies
\{ K_{n,m}\}_n \GLT \kappa_m:=
\sum_{p=-m}^{m}f_p(x) e^{\textnormal{i} p\theta}.
\]
An  argument similar to the one used for Lemma \ref{res:known} tells us that $\kappa_m$ is also a spectral symbol for $\{ K_{n,m}\}_n$. Since $\kappa_m\to\kappa$ almost everywhere, and $K_{n,m}$ is a truncation of $K_n$, it is natural to wonder whether a result like \textbf{GLT 5} is applicable in this situation to conclude that $\serie K\sim_{GLT,\sigma,\lambda} \kappa$, as reported in Conjecture \ref{conj:ALIF_1}.

It turns out that the result actually holds. The proof relies on several technical lemmata on a.c.s. convergence of certain matrix sequences, and some new results on spectral symbols: in order
to improve the readability of the paper, these are collected in Appendix \ref{sec:app}.

\begin{theorem}\label{teo}
	Let
	\[
	K_n = \left[ \frac{k((i-j)/L(x_i))}{L(x_i)}  \right]_{i,j=0}^{n-1},
	\]
	where
	$x_i=\frac i{n-1}$  and
	\begin{itemize}
		\item $k:\f R\to \f R$ is an even, non-negative, bounded measurable function, supported on $[-1,1]$,
		\item $L:[0,1]\to \f R$ is a non-negative function.
	\end{itemize}
	Suppose that
	\[
	f_p(x) := \frac{k(p/L(x))}{L(x)}
	\]
	is Riemann-Integrable for every $p\in \f Z$. Then,
	$$\serie K\sim_{GLT,\sigma,\lambda} \kappa(x,\theta):=
	\frac{1}{L(x)}\sum_{p\in\f Z}k\left(\frac{p}{L(x)}\right) e^{\textnormal{i} p\theta}=
	\sum_{p\in\f Z}f_p(x)e^{\textnormal{i} p\theta}
	.
	$$
\end{theorem}
\begin{proof}	
	Observe that the
	matrix $K_n$ can be rewritten as
	\[
	K_n =  \sum_{p\in \f Z}D'_n(f_p)T_n(e^{\textnormal i p\theta}),
	\]
	where
	\begin{equation}
	D'_n(f_p):= \diag([f_p(x_i)]_{i=0,\dots,n-1}).
	\end{equation}
	From Lemma 	\ref{res:diag_GLT} and \textbf{GLT 3,4} we know that
	\[
	K_{n,m} := \sum_{p=-m}^{m} D_n'(f_p)T_n(e^{\textnormal{i} p\theta})\implies
	\{ K_{n,m}\}_n \GLT \kappa_m:=
	\sum_{p=-m}^{m}f_p(x) e^{\textnormal{i} p\theta}.
	\]
	Notice that $\kappa_m\to \kappa$ in measure, since it converges pointwise.
	As a consequence, if we prove that
	\[
	\lim_{m\to \infty} \limsup_{n\to\infty} \frac 1n \| K_n-K_{n,m}\|_2^2 = 0,
	\]
	then Theorem \ref{res:acs_crit} guarantees us that $\{ K_{n,m}\}\acs \serie K$ and \textbf{GLT 5} says that $\serie K\GLT \kappa$.
	Eventually, since $f_j = f_{-j}$ are real valued functions, $K_{n,m}$ can be written as
	\[
	K_{n,m} =
	D_n'(f_0) +
	\sum_{p=1}^{m} \left[D_n'(f_p)T_n(e^{\textnormal{i} p\theta})
	+ D_n'(f_p)^*T_n(e^{\textnormal{i} p\theta})^*\right].
	\]
	so	we can use, in order,
	Theorem   \ref{res:fund_alm_her},
	Lemma \ref{res:convergence_alm_Her}
	and Lemma \ref{res:GLT_alm_Her} to conclude that
	$$\serie K\sim_{GLT,\sigma,\lambda} \kappa(x,\theta).$$
	
	$ $\\
	Let us then estimate $\| K_n-K_{n,m}\|_2^2$. Notice that if $p\ne 0$, then $ k(p/L(x))\ne 0 \implies L(x)\ge p$, so
	\begin{align*}
	\| K_n-K_{n,m}\|_2^2 &=
	\sum_{i,j} ( f_{i-j}(x_{i-1}) - f_{i-j} (x_{i-1})\chi_{|i-j|\le m}    )^2
	\\
	&=
	\sum_{i,j}  \frac{k((i-j)/L(x_{i-1}))^2}{L(x_{i-1})^2}\chi_{|i-j|> m}   \\
		&\le
		\|k\|_\infty^2
	\sum_{i,j}  \frac{\chi_{|i-j|> m}}{(i-j)^2}   \\
	&\le
	2\|k\|_\infty^2 n
	\sum_{p= m+1}^{\infty }  \frac 1{p^2}
	\end{align*}
As a consequence, we conclude that
\begin{align*}
\lim_{m\to \infty} \limsup_{n\to\infty} \frac 1n\| K_n-K_{n,m}\|_2^2 &\le
\lim_{m\to \infty} \limsup_{n\to\infty}
2\|k\|_\infty^2
\sum_{p= m+1}^{\infty }  \frac 1{p^2} =
2\|k\|_\infty^2 \lim_{m\to \infty}\sum_{p= m+1}^{\infty }  \frac 1{p^2} =0.
\end{align*}
\end{proof}

 With an analogous proof one can see that conjecture is true even if  $f_p$ are just continuous a.e.

\section{Conjecture 2}\label{sec:Conj2}

The statement of Conjecture 2 in itself is ambiguous, since in the ALIF method it has never been specified a way to choose the length function $\ell(x)$ that is needed to build the filter and the iteration matrix $K_n$. This step is fundamental for the convergence of the method, and it is easy to build examples where a poor choice of $\ell(x)$ lead to an infinite loop, for almost any initial input $\bm r$ (in particular, for any $\bm r$ with more than two extrema).

\begin{example}\label{ex:counterexample_2}
If we require that $k(0) = 1$, $k(\pm 1) = 0$ (for example, $k(x) = \chi_{[-1/2,1/2]}$ or $k(x) = 1-|x|$), and take $L(x) \equiv 1$, then it is evident that $K_n=I_n$ and $\kappa(x,\theta) \equiv 1$, so the conditions of Conjecture 5.4 are met.
In this case, the Discrete ALIF iteration yields
$\bm g_2 = (I_n-K_n)\bm r = 0$ for every initial signal $\bm r$, so it can't converge, since the number of extrema of $\bm r$ never changes.
\end{example}

The example shows that the knowledge of $\kappa(x,\theta)$ is not enough to conclude whether the method converges. Nonetheless, as shown in Theorem \ref{teo},  it  provides some information on the convergence of the inner loop, since it begets an approximation of the eigenvalues of $K_n$. In fact, we can observe that in Example \ref{ex:counterexample_2}, the inner loop always converges. As a consequence, we can reinterpret the Conjecture as follows:

\begingroup
\renewcommand\theconjecture{2'}
\begin{conjecture}\label{conj:ALIF_2'}
	Assuming the hypotheses of Lemma \ref{res:known}, and \eqref{eq:necessary_condition}, the inner loop of Discrete ALIF converges.
\end{conjecture}
\endgroup


%
%

Sadly, it is possible to build a counterexample where $\rho(I_n-K_n)>1$, that leads to a diverging inner loop.
Consider the matrix
\begin{equation}\label{eq:Kn}
	K_n =
	\begin{pmatrix}
	0.7 &   0.48   & 0.15\\
	0.34 &   0.38  &  0.34\\
	0.24  &  0.41 &   0.49
	\end{pmatrix},
\end{equation}
that has  negative determinant $-0.00081$, and thus it has a negative eigenvalue $\lambda \sim -0.0018$ and $\rho(I_n-K_n) \sim 1.0018>1$. Every row can be seen as the coefficients of a nonnegative trigonometric polynomial, bounded by $2$. In particular,
\begin{itemize}
	\item $f_1(\theta) = 0.7  +	0.96 \cos(\theta) +	0.3 \cos(2\theta)$
	\item 								
	$f_2(\theta) =0.38 +0.68 \cos(\theta) +	0.496\cos(2\theta) +	0.288\cos(3\theta)+ 	0.124\cos(4\theta)
	+ 	0.032\cos(5\theta)	$
	\item 				
	$f_3(\theta) = 0.49 + 0.82 \cos(\theta) +	 0.48\cos(2\theta) + 0.18\cos(3\theta)+ 	0.03\cos(4\theta)$
\end{itemize}
In fact, the maximum of each function is attained at $\theta=0$, where
\[
f_1(0)=1.96,\qquad f_2(0)=2,\qquad f_3(0)=2.
\]
Moreover, if we substitute $y=\cos(\theta)$, then
\begin{align*}
	f_1(\theta) &= 0.7  +	0.96 \cos(\theta) +	0.3 \cos(2\theta)\\
	&= ( 225y^2 + 360y + 150)
	/375\\
	&= (15y + 12)^2/375 + 6/375\ge   6/375>0,\\
	f_2(\theta) &= 0.38 +0.68 \cos(\theta) +	0.496\cos(2\theta) +	0.288\cos(3\theta)+ 	0.124\cos(4\theta)
	+ 	0.032\cos(5\theta)\\
	&= (64y^5 + 124y^4 + 64y^3 - 3y + 1)/125\\
	&= (y+1)(64y^4 + 60y^3 +4y^2 -4y +1 )/125\\
	&= (y+1)(
	(32y^2 + 15y-3)^2 +
	(31y^2 + 26y + 7) )/2000 \ge 0,\\
	f_3(\theta) &= 0.49 + 0.82 \cos(\theta) +	 0.48\cos(2\theta) + 0.18\cos(3\theta)+ 	0.03\cos(4\theta)\\
	&=
	(6y^4 + 18y^3 + 18y^2 + 7y + 1)/25\\
	&=(y+1)
	( (y + 1/3)^2(6y+8) + 1/9)/25\ge 0,
\end{align*}
where $31y^2 + 26y + 7 > 0$ for every $y$, and
$-1\le y\le 1$ implies that $y+1\ge 0$ and $6y+8\ge 2>0$. \\

We want to find $k(x)$ and $L(x)$ that induce the matrix $K_n$. Recall that
 $k(x)$ must be an even, non-negative, bounded, measurable function with $\|k\|_1 =1$, and  compactly supported on $[-1,1]$ and the function $L(x)$ must be  strictly positive on $[0,1]$.

 In this case, $n=3$ and  $x_0 =0$, $x_1 = 1/2$, $x_2=1$, so if  we impose $L(x)$ to be a step function that takes only three values $L(x_0), L(x_1), L(x_2)$, then for every $x\in[0,1]$,  the symbol
 \[
 \kappa(x,\theta)
 := \frac{1}{L(x)}\sum_{j\in\f Z}k\left(\frac{j}{L(x)}\right) e^{\textnormal{i} j\theta}.
 \]
 is equal to $\kappa(x_i,\theta)$ for some $i=0,1,2$.  Notice that, from
the definition of $K_n$
	\[
K_n = \left[ \frac{k((i-j)/L(x_i))}{L(x_i)}  \right]_{i,j=0}^{n-1},
\]
if we impose
\begin{equation}\label{eq:symbol_building}
	\kappa(x_i,\theta) = f_{i+1}(\theta), \qquad i=0,1,2,
\end{equation}
then automatically
\[
\kappa(x,\theta) = \kappa(x_i,\theta) = f_{i+1}(\theta)\implies
0\le \kappa(x,\theta) \le 2, \quad  (x,\theta) \in [0,1]\times[-\pi,\pi]
\]
and $K_n$ takes the form in \eqref{eq:Kn}. Moreover, since $L(x)$ is a step function, the hypotheses of Conjecture \ref{conj:ALIF_2'} hold.
From \eqref{eq:symbol_building}, we need to equate the Fourier coefficients in $\theta$, and since $k$ needs to be an even function, it is sufficient to impose
\[
\frac{1}{L(x_i)}k\left(\frac{j}{L(x_i)}\right) = (f_{i+1})_{j}
, \qquad i=0,1,2,\qquad j\in \mathbb N,
\]
where $(f_{i+1})_{j}$ is the $j$-th Fourier coefficient  of $f_{i+1}$.
Taking the values
\[
L(0) = 3, \quad
L(1/2) = \frac {105}{19}, \quad
L(1) = \frac {30}{7}
\]
and $k(0) = 21/10$,  we have
\begin{align*}
\frac{1}{L(0)}k\left(\frac{j}{L(0)}\right) &= (f_{1})_{j} = 0, \qquad j>2,\\
\frac{1}{L(1/2)}k\left(\frac{j}{L(1/2)}\right) &= (f_{2})_{j} = 0, \qquad j> 5,\\
\frac{1}{L(1)}k\left(\frac{j}{L(1)}\right) &= (f_{3})_{j} = 0, \qquad j>4,
\end{align*}
since $k$ has support on $(-1,1)$, and
\[
k(0)/L(0) = (f_1)_0 = 7/10,\qquad
k(0)/L(1/2) = (f_2)_0 = 19/50,\qquad
k(0)/L(1) = (f_3)_0 = 49/100.
\]
The remaining conditions are reported in the following table:
\begin{table}[h!]
	\centering
\small{
	\begin{tabular}{ |c||c|c|c|c|c|c|c|c|c|c|c| }
		\hline
		$x$ &  19/105& 7/30& 1/3& 38/105& 7/15& 19/35& 2/3& 7/10& 76/105& 19/21& 14/15\\
		\hline
		$k(x)$ & 	357/190& 123/70& 36/25& 651/475& 36/35& 378/475& 9/20& 27/70& 651/1900& 42/475& 9/140		\\
		\hline
	\end{tabular}
}
	\caption{Conditions on the filter $k(x)$.}
	\label{tab:t1}
\end{table}

\begin{figure}[h]
	\makebox[\textwidth][c]{\includegraphics[width=\textwidth]{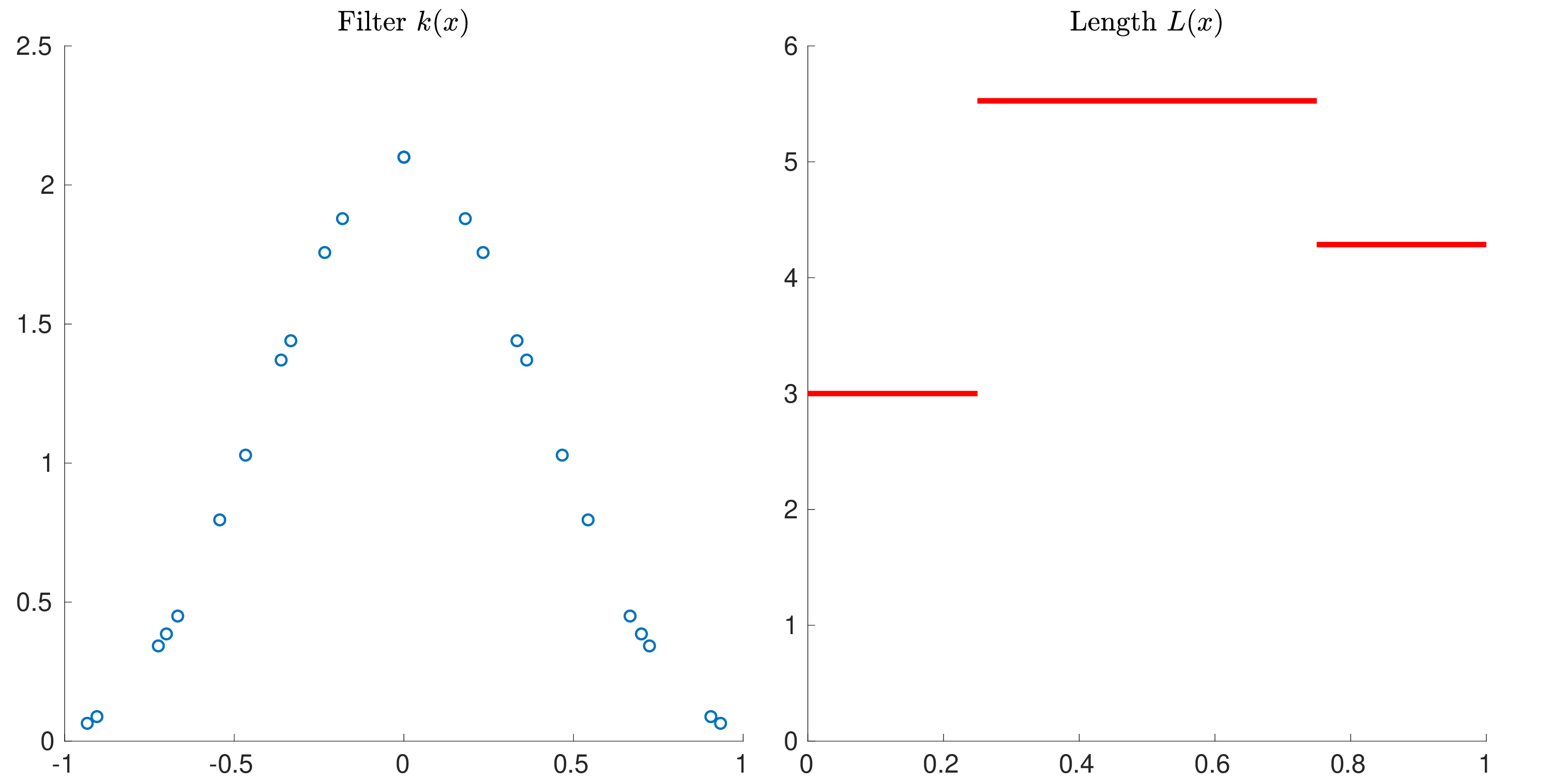}}%
	\caption{On the left, the conditions on the filter $k(x)$. On the right, the step-function $L(x)$.}
	\label{fig:f1}
\end{figure}

Since all the points where $k$ is evaluated are distinct, we can find an  even, non-negative, bounded and continuous measurable function $k$ supported on $(-1,1)$, that respects all conditions. The most simple example is a piecewise linear function connecting all the conditions shown in Figure \ref{fig:f1}. In this case, $\|k\|_1>1$, but notice that    the filter $k'(x) := k(x)/\|k\|_1$ and the same length function $L(x)$ produce the matrix $K'_n = K_n/\|k\|_1$, that is still a counterexample to Conjecture \ref{conj:ALIF_2'}, since it has a negative eigenvalue, and $0\le \kappa'(x,\theta) = \kappa(x,\theta) /\|k\|_1\le 2/\|k\|_1<2$.

\section{Conclusions}

In this work we tackle the open problems and conjectures left unsolved in \cite{cicone2019spectral}, which regard the convergence of the ALIF method.

In particular, we first review basic and fundamental properties of sequences of matrices, with particular emphasis on the GLT sequences, approximating classes of sequences of matrices and their spectral properties. Then we recall the ALIF method, its known properties and the two conjectures proposed in \cite{cicone2019spectral}.
In Theorem \ref{teo} we prove that Conjecture \ref{conj:ALIF_1} actually holds true. To achieve this result, we rely on several new technical lemmata on approximating classes of sequences
convergence of certain matrix sequences, namely the Almost-Hermitian Sequences and Almost-Hermitian GLT Sequences, and some new results on spectral symbols.
On the other hand, we show by counterexample that Conjecture \ref{conj:ALIF_2} cannot hold as it is. In particular we are able to show that the current formulation of the conjecture is too loose. We propose, then, a tighter formulation as Conjecture \ref{conj:ALIF_2'}. However, even in this case we are able to find a counterexample to the statement.

It remains an open problem if the ALIF algorithm can be proved to be convergent at all. We plan to study this problem in a future work.

\section*{Acknowledgements}

Antonio Cicone is a member of the Italian ``Gruppo Nazionale di Calcolo Scientifico'' (GNCS) of the Istituto Nazionale di Alta Matematica ``Francesco Severi'' (INdAM). He thanks the Italian Space Agency for the financial support under the contract ASI  "LIMADOU scienza" n$^{\circ}$ 2016-16-H0.

\newpage
\appendix
\section{Appendix: Technical results}\label{sec:app}
In this appendix, we provide the full details on some technical steps that are necessary for our analysis. For convenience, we have split the appendix into various subsections, according to the specific nature of the results contained therein.
Note that in the following, if $A$ is any subset of $[0,1]$, then $A^C$ is its complement set in $[0,1]$.

\subsection{Auxiliary Result}

\begin{lemma}\label{res:aux_osc}
	Let $a:[0,1] \to \f C$ be a bounded function and call $\alpha$-oscillation of $a$ the function
	\[
	\omega_\alpha(x) := \sup_{z\in B_\alpha(x) \cap [0,1]} |a(z)|  -\inf_{z\in B_\alpha(x) \cap [0,1]} |a(z)|,
	\]
	where $B_\alpha(x)$ is the open ball with centre $x$ and radius $\alpha$.
	If $a$ is continuous a.e., then also $\omega_\alpha$ is continuous a.e.
\end{lemma}
\begin{proof}
	Let $E$ be the set of discontinuity points for $a$, and define
	\[
	Z:= \set{ x\in [0,1] | x-\alpha\in E \text{ or } x+\alpha\in E }.
	\]
	Note that the measure of $Z$ is at most two times the measure of $E$, so it is zero.
	Let now $x\in Z^C$ and $a(x) = b$. We know that both $x-\alpha$ and $x+\alpha$ (when they are inside $[0,1]$) are continuity points for $a$, so given
	any $\ve >0$,
	there exists $\delta>0$ such that
	\[
	|a(x-\alpha) - a(x-\alpha + y)| \le \ve, \quad
	|a(x+\alpha) - a(x+\alpha + y)|\le \ve\qquad \forall |y|<\delta.
	\]
	As a consequence, for every $0<y<\delta$ the following holds
	\begin{align*}
	\sup_{z\in B_\alpha(x+y) \cap [0,1]} |a(z)|
	&=
	\max\left\{
	\sup_{z\in (x+y-\alpha,x+\alpha) \cap [0,1]} |a(z)| ,
	\sup_{z\in [x+\alpha,x+\alpha+y) \cap [0,1]} |a(z)|
	\right\}\\
	\implies
	\sup_{z\in B_\alpha(x+y) \cap [0,1]} |a(z)|
	&\le
	\max\left\{
	\sup_{z\in B_\alpha(x) \cap [0,1]} |a(z)| ,
	|a(x+\alpha)| +\ve
	\right\}
	\le \sup_{z\in B_\alpha(x) \cap [0,1]} |a(z)| +\ve, \\
	\sup_{z\in B_\alpha(x+y) \cap [0,1]} |a(z)|
	&\ge
	\sup_{z\in (x+y-\alpha,x+\alpha) \cap [0,1]} |a(z)|
	\ge
	\sup_{z\in B_\alpha(x) \cap [0,1]} |a(z)| -\ve ,
	\\
	\inf_{z\in B_\alpha(x+y) \cap [0,1]} |a(z)|
	&=
	\min\left\{
	\inf_{z\in (x+y-\alpha,x+\alpha) \cap [0,1]} |a(z)| ,
	\inf_{z\in [x+\alpha,x+\alpha+y) \cap [0,1]} |a(z)|
	\right\}\\
	\implies
	\inf_{z\in B_\alpha(x+y) \cap [0,1]} |a(z)|
	&\ge
	\min\left\{
	\inf_{z\in B_\alpha(x) \cap [0,1]} |a(z)| ,
	|a(x+\alpha)| -\ve
	\right\}
	\ge \inf_{z\in B_\alpha(x) \cap [0,1]} |a(z)| -\ve,\\
	\inf_{z\in B_\alpha(x+y) \cap [0,1]} |a(z)|
	&\le
	\inf_{z\in (x+y-\alpha,x+\alpha) \cap [0,1]} |a(z)|\le   \inf_{z\in B_\alpha(x) \cap [0,1]} |a(z)| +\ve.
	\end{align*}
	\begin{equation}\label{supinf}
	\left| 	\sup_{z\in B_\alpha(x+y) \cap [0,1]} |a(z)| - \sup_{z\in B_\alpha(x) \cap [0,1]} |a(z)|\right|\le \ve, \qquad
	\left| 	\inf_{z\in B_\alpha(x+y) \cap [0,1]} |a(z)| - \inf_{z\in B_\alpha(x) \cap [0,1]} |a(z)|\right|\le \ve.
	\end{equation}
	With an analogous argument we can show that \eqref{supinf} holds also for $-\delta<y<0$ and even if $x+\alpha$ or $x-\alpha$ are not inside $[0,1]$, so
	\[
	|\omega_\alpha(x) -\omega_\alpha(x+y)|
	\le
	\left| \sup_{z\in B_\alpha(x) \cap [0,1]} |a(z)|
	-\sup_{z\in B_\alpha(x+y) \cap [0,1]} |a(z)|
	\right|
	+
	\left|
	\inf_{z\in B_\alpha(x+y) \cap [0,1]} |a(z)|
	-\inf_{z\in B_\alpha(x) \cap [0,1]} |a(z)|
	\right|\le 2\ve.
	\]
	This is enough to prove that  $\omega_\alpha$ is continuous at every point of $Z^C$, meaning it is continuous a.e.
\end{proof}

\begin{lemma}\label{res:diag_GLT}
	Let $x_i=\frac i{n-1}$ and
	$$D'_n(a):= \diag([a(x_i)]_{i=0,\dots,n-1})$$
	for any Riemann-Integrable function $a:[0,1]\to \f C$. Then
	\[
	\{ D'_n(a)  \}_n \GLT a(x).
	\]
\end{lemma}
\begin{proof}
	First of all, let us prove it in the case $a(x)$ continuous. Notice that
	\[
	\left| \left[ D'_n(a) - D_n(a) \right]_{i,i} \right|
	=
	|a(x_{i-1}) - a(i/n)| \le \omega_a\left(\left|\frac {i-1}{n-1}-\frac in\right|\right) =
	\omega_a\left(\left|\frac {n-i}{n(n-1)}\right|\right)\le
	\omega_a(1/n)
	\]
	where $\omega_a$ is the continuity modulus of $a(x)$. Since $\omega_a(x)\xrightarrow{x\to 0}0$, we obtain that $\| D'_n(a) - D_n(a)\| \xrightarrow{n\to\infty}0$ and in particular, from Proposition \ref{res:zero-dis}, we know that $\{D_n(a) - D_n'(a) \}_n$ is zero-distributed. The thesis follows from \textbf{GLT 3} and \textbf{GLT 4}.\\
	
	Suppose now that $a(x)$ is Riemann-Integrable. From the density of the continuous function in $L^1([0,1])$, we know that there exists a
	 sequence of continuous functions $a_m(x)$ converging  in $L^1([0,1])$ to $a(x)$. Notice that $a-a_m$ is Riemann-Integrable for any $m$.
	\begin{align*}
	\lim_{m\to\infty}\limsup_{n\to\infty} \frac 1n\|D_n'(a_m) - D'_n(a)\|_1 &=
	\lim_{m\to\infty}\limsup_{n\to\infty}\frac 1n\sum_{i=0}^{n-1}
	\left| a_m(x_i) - a(x_i)\right|\\
	&=\lim_{m\to\infty} \int_0^1 |a_m(x)-a(x)| \,{{\rm d}}x =0,
	\end{align*}
	and Theorem \ref{res:acs_crit} shows that
	$$
	\{D_n'(a_m) \}_n\acs \{ D_n'(a) \}_n.
	$$
	The thesis follows from \textbf{GLT 5}.
\end{proof}

\subsection{Almost-Hermitian Sequences}

From now on, we say that a sequence $\serie A$ is almost-Hermitian if there exists an Hermitian sequence $\serie{\wt A}$ such that $\|A_n-\wt A_n\|_2=o(\sqrt n)$.

\begin{lemma}\label{res:GLT_alm_Her}
	Suppose $\serie A\GLT k$ is an almost-Hermitian sequence. In this case, $k$ is real valued and
	\[
	\serie A\sim_{GLT,\sigma,\lambda} k.
	\]
\end{lemma}
\begin{proof}
	Since $A_n = \wt A_n + (A_n - \wt A_n)$ where $\wt A_n$ is Hermitian and $\|A_n-\wt A_n\|_2=o(\sqrt n)$, from
	\textbf{GLT 2} we conclude that $\serie A\sim_\lambda k$.
\end{proof}

\begin{lemma}\label{res:alm_Her_vec_space}
	The set of almost-Hermitian sequences is a real vectorial space.
\end{lemma}
\begin{proof}
	If $\serie A$ is an almost-Hermitian sequence and $c\in \f R$, then
	\[
	\|cA_n-c\wt A_n\|_2=|c| \|A_n-\wt A_n\|_2=o(\sqrt n).
	\]
	If $\serie B$ is also almost-Hermitian, then
	\[
	\|A_n + B_n -\wt A_n-\wt B_n\|_2 \le \|A_n-\wt A_n\|_2 + \|B_n-\wt B_n\|_2 =o(\sqrt n).
	\]
\end{proof}

\begin{lemma}\label{res:convergence_alm_Her}
	Given a sequence of almost-Hermitian sequences $\{ B_{n,m}  \}_n$   suppose that there exists a sequence $\serie B$ with
	\[
	\lim_{m\to \infty} \limsup_{n\to\infty}\frac 1n \| B_{n,m} - B_n\|_2^2  = 0.
	\]
	In this case, $\serie B$ is almost-Hermitian.
\end{lemma}
\begin{proof}
	From the definition of almost-Hermitian sequences, we can find $\wt B_{n,m}$ Hermitian matrices with
	$\|B_{n,m} - \wt B_{n,m} \|_2 = o(\sqrt n)$.	
	Let us now estimate the norm of the imaginary part of $\serie B$.
	\begin{align*}
	\|\Im(B_n) \|_2 &= \frac 12 \| B_n - B_n^*\|_2 \\
	&\le \frac 12 \left(
	\|B_n-B_{n,m}\|_2 +
	\|B_{n,m} - \wt B_{n,m}\|_2 	+
	\|\wt B_{n,m} - B_{n,m}^*\|_2 + \|B_{n,m}^* - B_n^*\|_2
	\right) \\
	&= \|B_n-B_{n,m}\|_2  +\|B_{n,m}  - \wt B_{n,m}\|_2\\
	\implies
	\limsup_{n\to\infty}\frac 1{\sqrt n}\|\Im(B_n) \|_2 &=
	\lim_{m\to \infty} \limsup_{n\to\infty} \frac 1{\sqrt n}\|\Im(B_n) \|_2 \\
	&\le
	\lim_{m\to \infty} \limsup_{n\to\infty} \frac 1{\sqrt n}
	\|B_n-B_{n,m}\|_2  +\frac 1{\sqrt n}\|B_{n,m}  - \wt B_{n,m}\|_2 = 0
	\\
	\implies \|\Im(B_n) \|_2&= o(\sqrt n).
	\end{align*}
	Since $B_n = \Re(B_n)  +\textnormal i\Im(B_n)$, the sequence $\serie B$ is an Hermitian sequence $\{\Re(B_n)\}_n$ plus a $o(\sqrt n)$ correction, thus it is an almost-Hermitian sequence.
\end{proof}

\subsection{Almost-Hermitian GLT Sequences}

The following result is formulated so that it can be applied to the problem at hand, but the same argument works also with $D_n(f_p)$ instead of $D_n'(f_p)$.

\begin{theorem}\label{res:fund_alm_her}
	Given any Riemann Integrable function $a:[0,1]\to\f C$ and any natural number $m$, denote
	\[
	A_n(a,m) := D_n'(a) T_n(e^{\textnormal im\theta}) + D_n'(a)^*  T_n(e^{\textnormal im\theta})^*,
	\]
	where
	$$D'_n(a):= \diag([a(x_i)]_{i=0,\dots,n-1})$$
	and $x_i=\frac i{n-1}$.
	In this case, if $a_0,a_1,\dots,a_{p-1}$ are Riemann Integrable functions, then
	\[
	\left\{ \frac 12 A_n(a_0,0) + \sum_{m=1}^{p-1}A_n(a_m,m) \right\}_n
	\]
	is an almost-Hermitian sequence for every positive number $p$.
\end{theorem}
\begin{proof}
	First of all, 	
	from \textbf{GLT 3,4} and Lemma \ref{res:diag_GLT}, we know that
	\[
	\{A_n(a,m) \}_n\GLT 2\Re(a(x)e^{\textnormal im\theta}).
	\]
	If $m=0$, then $A_n(a,m)$ is Hermitian for every $n$, so the thesis follows.
	Suppose now that $m>0$ and define the Hermitian matrix $\wt A_n(a,m)$ as
	\[
	[\wt A_n(a,m)]_{i,j} =
	\begin{cases}
	a(x_{j-1}), & i-j = m,\\
	\ol a(x_{i-1}), & i-j = -m,\\
	0, &\text{otherwise.}
	\end{cases}
	\]
	Let $Z_n = A_n(a,m) -\wt A_n(a,m)$, and notice that
	\[
	[Z_n]_{i,j} =
	\begin{cases}
	a(x_{i-1}) - a(x_{j-1}), & i-j = m,\\
	0, &\text{otherwise.}
	\end{cases}
	\]
	Notice that if $\omega_\alpha$ is the $\alpha$-oscillation relative to $a$, defined as
	\[
	\omega_\alpha(x) := \sup_{z\in B_\alpha(x) \cap [0,1]} |a(z)|  -\inf_{z\in B_\alpha(x) \cap [0,1]} |a(z)|,
	\]
	then $\omega_\alpha \xrightarrow{\alpha\to 0} 0$ a.e. since
	\[
	a \text{ continuous on }x\implies \omega_\alpha(x) \xrightarrow{\alpha\to 0} 0.
	\]
	We can thus fix $\ve >0$ and find $\alpha_\ve$ (that we call $\alpha$ for simplicity) such that
	\[
	E:= \set{x\in [0,1] : \omega_\alpha(x) \ge\ve   }, \quad 	\mu(E)  <\ve/2.
	\]
	Notice that $a$ is bounded and continuous a.e., so by Lemma \ref{res:aux_osc}, $\omega_\alpha$ is also continuous a.e. and thus $E^C$ is an open set up to a negligible set.
	Every open set can be approximated from the inside by a finite union of open intervals, so we can take $G\cu E^C$ a finite union of open intervals with measure $\mu(G)>\mu(E^C)- \ve/2>1-\ve$. We can approximate the measure of $G$ as
	\[
	\lim_{n\to\infty} \frac 1n \#\set{ i |  0\le i\le n-1, x_i \in G  } =\mu (G) >1-\ve.
	\]
	Let $N$ be an index such that $m/N <\alpha$ and
	\[
	\frac 1n \#\set{ i |  0\le i\le n-1, x_i \in G  }  >1-2\ve \quad \forall n>N.
	\]
	If $\|a\|_\infty = M$ and $n>N$, we have
	\begin{align*}
	\|Z_n\|_2^2 &= \sum_{j=1}^{n-m} |a(x_{j+m-1}) - a(x_{j-1})|^2\\
	&= \sum_{j\le n-m}^{x_{j-1}\not\in G} |a(x_{j+m-1}) - a(x_{j-1})|^2 + \sum_{j\le n-m}^{x_{j-1}\in G} |a(x_{j+m-1}) - a(x_{j-1})|^2\\
	&\le
	\#\set{ i |  0\le i\le n-1, x_i \in G^C  } \cdot  4M^2 +
	\sum_{j\le n-m}^{x_{j-1}\in E^C} \omega_\alpha(x_{j-1})^2
	\\
	&\le  8\ve n M^2 + n \ve^2.
	\end{align*}
	As a consequence
	\[
	\limsup_{n\to\infty} \frac 1n\|Z_n\|_2^2 \le 8\ve  M^2 +  \ve^2
	\]
	for every $\ve >0$, so
	\[
	\limsup_{n\to\infty} \frac 1n\|Z_n\|_2^2 = 0.
	\]
	We have thus shown that $\{ A_n(a,m) \}_n$ is  almost-Hermitian, and Lemma \ref{res:alm_Her_vec_space} let us conclude that 	
	\[
	\left\{ \frac 12 A_n(a_0,0) + \sum_{m=1}^{p-1}A_n(a_m,m) \right\}_n
	\]
	is also an almost-Hermitian sequence for every positive number $p$.
\end{proof}

\end{document}